\documentclass[12pt,a4paper]{amsart}
\usepackage{amsfonts,color}
\usepackage{amsthm}
\usepackage{amsmath}
\usepackage{amscd}
\usepackage{amssymb}
\usepackage[latin2]{inputenc}
\usepackage{t1enc}
\usepackage[mathscr]{eucal}
\usepackage{indentfirst}
\usepackage{graphicx}
\usepackage{graphics}
\usepackage{pict2e}
\usepackage{epic}
\usepackage{url}
\usepackage{epstopdf}
\usepackage{comment}

\numberwithin{equation}{section}
\usepackage[margin=2.6cm]{geometry}

\allowdisplaybreaks

\theoremstyle{plain}
\newtheorem{Th}{Theorem}[section]
\newtheorem{Lemma}[Th]{Lemma}

\newtheorem{Prop}[Th]{Proposition}

 \theoremstyle{definition}
\newtheorem{Def}[Th]{Definition}
\newtheorem{Conj}[Th]{Conjecture}
\newtheorem{Rem}[Th]{Remark}
\newtheorem{?}[Th]{Problem}

\newcommand{\E}{\mathbb{E}}
\newcommand{\ch}{\mathrm{ch}}
\renewcommand{\P}{\mathbb{P}}

\begin{document}

\title[Forests and connected spanning subgraphs]{On the number of forests and connected spanning subgraphs}

\author[M. Borb\'enyi]{M\'arton Borb\'enyi}

\address{ELTE: E\"{o}tv\"{o}s Lor\'{a}nd University \\ H-1117 Budapest
\\ P\'{a}zm\'{a}ny P\'{e}ter s\'{e}t\'{a}ny 1/C}

\email{marton.borbenyi@gmail.com}

\author[P. Csikv\'ari]{P\'{e}ter Csikv\'{a}ri}

\address{Alfr\'ed R\'enyi Institute of Mathematics, H-1053 Budapest Re\'altanoda utca 13/15 \and ELTE: E\"{o}tv\"{o}s Lor\'{a}nd University \\ Mathematics Institute, Department of Computer
Science \\ H-1117 Budapest
\\ P\'{a}zm\'{a}ny P\'{e}ter s\'{e}t\'{a}ny 1/C}

\email{peter.csikvari@gmail.com}

\author[H. Luo]{Haoran Luo}

\address{Department of Mathematics, University of Illinois at Urbana-Champaign, Urbana, Illinois 61801, USA}

\email{haoranl8@illinois.edu}

\thanks{The first  author was partially supported by the EFOP program (EFOP-3.6.3-VEKOP-16-2017-00002) and the New National Excellence Program (\'UNKP) when the project started. 
The second author  is supported by the Counting in Sparse Graphs Lend\"ulet Research Group. When the project started he was also supported by the
 Marie Sk\l{}odowska-Curie Individual Fellowship grant no. 747430.
}



\begin{abstract} Let $F(G)$ be the number of forests of a graph $G$. Similarly let $C(G)$ be the number of connected spanning subgraphs of a connected graph $G$. We bound $F(G)$ and $C(G)$ for regular graphs and for graphs with a fixed average degree. Among many other things we  study $f_d=\sup_{G\in \mathcal{G}_d}F(G)^{1/v(G)}$, where $\mathcal{G}_d$ is the family of $d$-regular graphs, and $v(G)$ denotes the number of vertices of a  graph $G$. We show that $f_3=2^{3/2}$, and if $(G_n)_n$ is a sequence of $3$-regular graphs with the length of the shortest cycle tending to infinity, then $\lim_{n\to \infty}F(G_n)^{1/v(G_n)}=2^{3/2}$. We also improve on the previous best bounds on $f_d$ for $4\leq d\leq 9$. 

\end{abstract}

\maketitle

\section{Introduction}
For a graph $G=(V,E)$ let $T_G(x,y)$ denote its Tutte polynomial, that is,
$$T_G(x,y)=\sum_{A\subseteq E}(x-1)^{k(A)-k(E)}(y-1)^{k(A)-|A|-v(G)},$$
where $k(A)$ denotes the number of connected components of the graph $(V,A)$, and $v(G)$ denotes the number of vertices of the graph $G$. It is well-known that special evaluations of the Tutte polynomial have various combinatorial meaning. For instance, $T_G(1,1)$ counts the number of spanning trees for a connected  graph $G$. Similarly, $T_G(2,1)$ enumerates the number of forests (acyclic edge subsets), and for a connected graph $G$ the evaluation $T_G(1,2)$ is equal to the number of connected spanning subgraphs, and $T_G(2,0)$ is the number of acyclic orientations of $G$. In this paper, we use the notation $F(G)=T_G(2,1)$  for the number of forests,  $C(G)=T_G(1,2)$ for the number of connected spanning subgraphs, and $a(G)=T_G(2,0)$ for the number of acyclic orientations.

The scope of this paper is to give various upper bounds for $F(G)$ and $C(G)$ in terms of the average degree. A special emphasis is put on the case when $G$ is a regular graph of degree $d$. 

\subsection{Number of forests} First, we collect our results for the number of forests.

The following statement is well-known and serves as a motivation for many of our results. For the sake of completeness, we will give a proof of it.

\begin{Prop}[\cite{Thom}] \label{product-forest}
Let $G$ be a graph, and let $d_v$ be the degree of a vertex $v$. Then
$$F(G)\leq \prod_{v\in V(G)}(d_v+1).$$
\end{Prop}

When one applies Proposition~\ref{product-forest} to regular graphs of degree $3$ and $4$, it turns out to be rather poor since the trivial inequality in terms of the number of edges $e(G)$, that is, $F(G)\leq 2^{e(G)}$ gives stronger results. Indeed, for a $3$-regular graph this trivial inequality gives $F(G)^{1/v(G)}\leq 2\sqrt{2}$, while for a $4$-regular graph it gives $F(G)^{1/v(G)}\leq 4$. Surprisingly, this inequality cannot be improved for $3$-regular graphs as the following result shows. Let $g(G)$ be the length of the shortest cycle, which is called the girth of the graph $G$.

\begin{Th} \label{3-regular-forest}
Let $(G_n)_n$ be a sequence of $3$-regular graphs with girth $g(G_n)\to \infty$. Then
$$\lim_{n\to \infty}a(G_n)^{1/v(G_n)}=2\sqrt{2},$$
and
$$\lim_{n\to \infty}F(G_n)^{1/v(G_n)}=2\sqrt{2}.$$
In particular, $f_3=2\sqrt{2}$.
\end{Th}

Note that the large girth requirement is necessary in the following sense. Suppose that for a fixed $k$ and $\varepsilon$ the graph $G$ contains at least $\varepsilon v(G)$ edge-disjoint cycles of length at most $k$. Then
$$F(G)^{1/v(G)}\leq c(k,\varepsilon)2^{3/2},$$
where $c(k,\varepsilon)<1$. Indeed, if $C_1,C_2,\dots ,C_r$ are edge-disjoint cycles of length at most $k$, then
$$F(G)\leq \prod_{i=1}^r(2^{|C_i|}-1)\cdot 2^{e(G)-\sum_{i=1}^r|C_i|}=\prod_{i=1}^r\frac{2^{|C_i|}-1}{2^{|C_i|}
}\cdot 2^{e(G)}\leq \left(\frac{2^k-1}{2^k}\right)^{\varepsilon v(G)}\cdot 2^{e(G)}.$$
Hence
$$F(G)^{1/v(G)}\leq \left(\frac{2^k-1}{2^k}\right)^{\varepsilon}\cdot 2^{3/2}.$$
We remark that the ratio $\frac{F(G)}{2^{e(G)}}$ can be rather large for a $3$-regular graph. For instance, for the Tutte-Coxeter graph this ratio is roughly $0.728$. Note that this is a $3$-regular graph on $30$ vertices with girth $8$. It seems that for cages, that is, for regular graphs that have minimal size for a given degree and girth, this ratio can be quite large. This motivates the following question.

\begin{?} Let $\mathcal{G}_3$ be the family of $3$-regular graphs. Is it true that
$$\sup_{G\in \mathcal{G}_3}\frac{F(G)}{2^{e(G)}}=1?$$
\end{?}

Before we turn our attention to $4$-regular graphs, let us give one more general upper bound for the number of forests. Let us introduce the entropy function
$$H(x):=x\ln \left(\frac{1}{x}\right)+(1-x)\ln \left(\frac{1}{1-x}\right)$$
 with the usual convention $H(0)=H(1)=0$. In the proofs, we will often use the following inequality. If $n\leq m/2$, then
 $$\sum_{k=0}^n\binom{m}{k}\leq \exp \left(mH\left(\frac{n}{m}\right)\right).$$
 The proof of this inequality can be found in \cite{AlSp}.
\medskip

\begin{Prop} \label{average-forest} Let $G$ be a graph with average degree $\overline{d}$. If $\overline{d}\geq 4$, then
$$F(G)^{1/v(G)}\leq \exp\left(\frac{\overline{d}}{2}H\left(2/\overline{d}\right)\right).$$
\end{Prop}

This simple inequality is based on the rather trivial observation that a forest can have at most $v(G)-1$ edges, and so
$$F(G)\leq \sum_{k=0}^{v(G)-1}\binom{e(G)}{k}.$$
The problem with this bound is that if the average degree is exactly $4$, this is not much different from the trivial upper bound $2^{e(G)}=4^{v(G)}$. Already Merino and  Welsh \cite{MeWe} noted that they found rather challenging to improve on this trivial bound even for grids. Nevertheless, $4^{v(G)}$ is definitely not the best answer for $4$-regular graphs as the following theorem shows.

\begin{Th} \label{4-regular-forest}
Let $\mathcal{G}_4$ be the family of $4$-regular graphs. Then
$$\sup_{G\in \mathcal{G}_4}F(G)^{1/v(G)}<3.994.$$
\end{Th}

We have seen that for  $3$-regular graphs the quantity $F(G)^{1/v(G)}$ is asymptotically maximized by large girth graphs. A similar theorem was proved for the number of spanning trees by McKay \cite{McKay1} for regular graphs of degree $d$ for arbitrary $d$. Here we show that the same holds for $F(G)$ for any $d$ assuming a well-known conjecture about a certain negative correlation.

\begin{Conj}[\cite{GW,Pem}]\label{correlation-conjecture} Let $G$ be a graph and let $\textbf{F}$ be a random forest chosen uniformly from all the forests of $G$. Let $e,f\in E(G)$, then
$$\P(e,f\in \textbf{F})\leq \P(e\in \textbf{F})\P(f\in \textbf{F}).$$
\end{Conj}

Assuming Conjecture~\ref{correlation-conjecture} we can prove a result on forests of $2$-covers which then implies our claim about the large girth graphs.

Recall that a $k$-cover (or $k$-lift) $H$ of a graph $G$ is defined as follows. The vertex set of  $H$ is $V(H)=V(G)\times \{0,1,\dots, k-1\}$, and if $(u,v)\in E(G)$,  then we choose a perfect matching between the vertex set $L_u=\{(u,i)\ |\ 0\leq i\leq k-1\}$ and $L_v=\{(v,i)\ |\ 0\leq i\leq k-1\}$. If $(u,v)\notin E(G)$, then there are no edges between $L_u$ and $L_v$. Figure~\ref{2-lift-picture} depicts a 2-lift.

\begin{figure}[h!] \label{2-lift-picture}
\begin{center}
\scalebox{.4}{\includegraphics{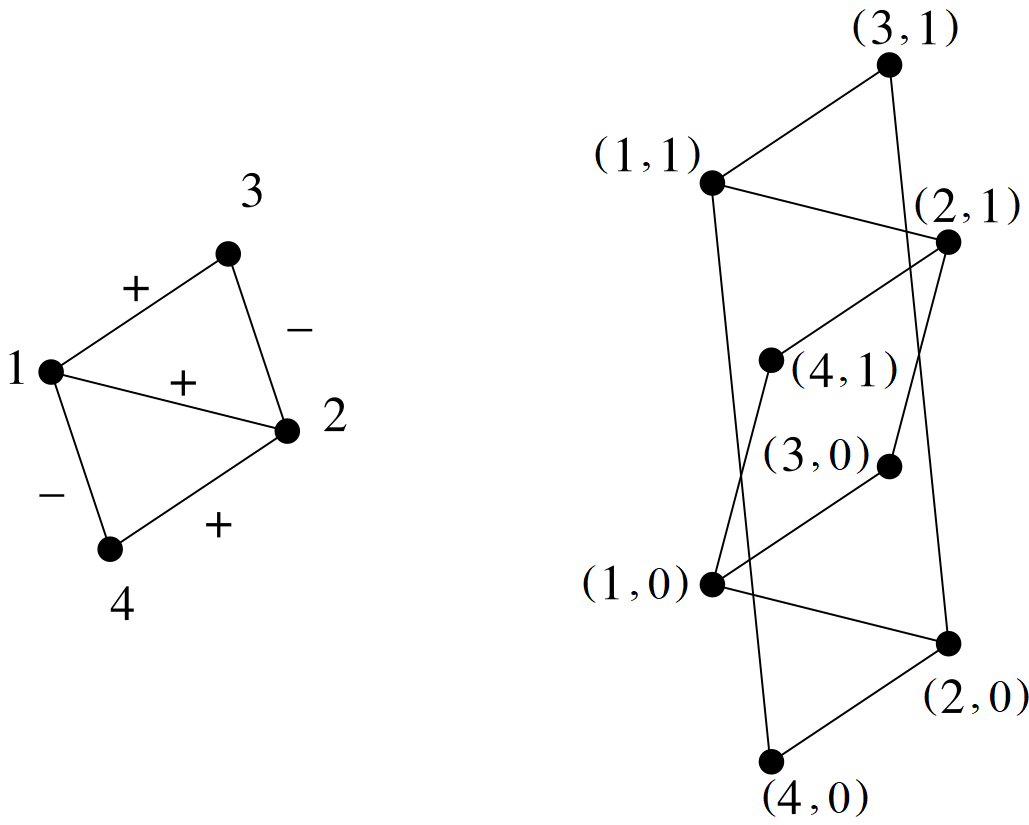}}
\caption{A $2$-lift.}
\end{center}
\end{figure}

When $k=2$ one can encode the $2$-lift $H$ by putting signs on the edges of the graph $G$: the $+$ sign means that we use the matching $((u,0),(v,0)),((u,1),(v,1))$ at the edge $(u,v)$,
the $-$ sign means that we use the matching  $((u,0),(v,1)),((u,1),(v,0))$ at the edge $(u,v)$. For instance, if we put $+$ signs to every edge, then we simply get the disjoint union $G\cup G$, and if we put $-$ signs everywhere, then the obtained $2$-cover $H$ is simply the tensor product $G\times K_2$. Observe that if $G$ is bipartite, then $G\cup G=G\times K_2$, but other $2$-covers might differ from $G\cup G$.

\begin{Th} \label{forest-cover}
Let $G$ be a graph, and let $H$ be a $2$-cover of $G$. If Conjecture~\ref{correlation-conjecture} is true, then we have
$$F(G\cup G)\leq F(H).$$
In other words, $F(G)^{1/v(G)}\leq F(H)^{1/v(H)}$.
\end{Th}

There is a nice property of covers that is related to the girth. For every graph $G$, there is a sequence of graphs $(G_n)_n$ such that $G_0=G$, $G_k$ is a $2$-cover of $G_{k-1}$, and  $g(G_k)\to \infty$. This is an observation due to Linial \cite{Lin}, his proof is also given in \cite{Csikv}. This observation and Theorem~\ref{forest-cover} (assuming Conjecture~\ref{correlation-conjecture}) together imply the following statement.
If one can prove that for any sequence of $d$-regular graphs $(G_n)_n$ with $g(G_n)\to \infty$, the limit $\lim_{n\to \infty}F(G_n)^{1/v(G_n)}$ always exists, and its value is (always) $s_d$, then $\sup_{G\in \mathcal{G}_d}F(G)^{1/v(G)}=s_d$.

Large girth $d$-regular graphs locally look like the infinite $d$-regular tree. So the above discussion suggests that it is natural to compare finite graphs with the infinite $d$-regular tree. At first sight, it might not be clear how to do it. Nevertheless, there is already such an argument in the literature. Kahale and Schulman \cite{KaSc} gave an upper bound on the number of acyclic orientations $a(G)$ in this spirit. Note that $a(G)=T_G(2,0)\leq T_G(2,1)=F(G)$. Their proof actually works for $F(G)$ too, and for $d\geq 6$ this upper bound is better than any of these three bounds: the trivial bound $2^{d/2}$ provided by $F(G)\leq 2^{e(G)}$, the bound $d+1$ provided by Proposition~\ref{product-forest}, and the bound $\exp\left(\frac{d}{2}H\left(\frac{2}{d}\right)\right)$ provided by Proposition~\ref{average-forest}.

\begin{Th}[Kahale and Schulman \cite{KaSc}] \label{KS-bound}
Let $G$ be a $d$-regular graph. Then
$$F(G)^{1/v(G)}\leq \frac{d+1}{\eta}\left(\frac{d-1}{d-\eta}\right)^{(d-2)/2},$$
where
$$\eta=\frac{(d+1)^2-(d+1)(d^2-2d+5)^{1/2}}{2(d-1)}.$$
\end{Th}

\noindent Theorem~\ref{KS-bound} gives the bound
$$F(G)^{1/v(G)}\leq d+\frac{1}{2}+\frac{1}{8d}+\frac{13}{48d^2}+O\left(\frac{1}{d^3}\right).$$

In this paper, we will review the proof of Theorem~\ref{KS-bound} and show how to improve on it for certain $d$.  The proof is actually a combination of the proof of Theorem~\ref{KS-bound} and  Proposition~\ref{average-forest}.
In particular, we will prove the following statement.

\begin{Th} \label{regular upper bounds}
Let $G$ be a $d$-regular graph, where $d\in \{5,6,7,8,9\}$. Then $F(G)^{1/v(G)}\leq C_d$, where $C_d$ is a constant strictly better than the one given in Theorem~\ref{KS-bound} and is given in Table~\ref{table forest}.
\end{Th}

\begin{center}
\begin{table}[h!] \label{table forest}
\caption{Bounds on the number of forests for small $d$}

\begin{tabular}{|c|c|c|c|c|}  \hline
$d$ & new bound $C_d$ & Thm.~\ref{KS-bound} & Prop.~\ref{product-forest} & Prop.~\ref{average-forest}\\ \hline
$5$   & $5.1965$           &   $5.5362$ & 6 & 5.3792\\ \hline
$6$   & $6.3367$           &   $6.5287$ & 7 & 6.7500         \\ \hline
$7$   & $7.4290$           &   $7.5236$ & 8 & 8.1169    \\ \hline
$8$   & $8.4843$           &   $8.5201$ & 9  & 9.4815      \\ \hline
$9$   & $9.5116$           &   $9.5174$ & 10  &  10.8447     \\ \hline
\end{tabular}

\end{table}
\end{center}

\begin{Rem}
One might wish to compare these results with existing bounds on finite and infinite sections of Archimedean lattices, cf. \cite{CMNN,ChSh1,ChSh2,MeWe,Mani}. For these specific graphs one may give more accurate bounds.
\end{Rem}

\subsection{Number of connected spanning subgraphs}

In this section, we collect the results on the number of connected spanning subgraphs. Again the trivial upper bound is $C(G)\leq 2^{e(G)}$ which gives $C(G)^{1/v(G)}\leq 2^{d/2}$ for a graph with average degree $d$.
This time this inequality can never be tight, not even for $3$-regular graphs.

\begin{Th} \label{regular-connected} Let $\mathcal{G}_d$ be the set of $d$-regular graphs. Then
$$\sup_{G\in \mathcal{G}_d}C(G)^{1/v(G)}<2^{d/2}\left(1-\frac{1}{2^d}\right)\exp\left(\frac{d}{2^{d}(2^d-1)}\right).$$
\end{Th}

We will again prove another upper bound for graphs with small average degree.

\begin{Th} \label{average-connected} Let $G$ be a graph with average degree $\overline{d}$. If $2< \overline{d}\leq 4$, then
$$C(G)\leq \frac{2}{\overline{d}-2}\exp\left(v(G)\cdot \frac{\overline{d}}{2}H\left(2/{\overline{d}}\right)\right).$$
\end{Th}

\subsection{This paper is organized as follows.} Each section of the paper contains a proof of a theorem or proposition that is stated in the introduction. The proofs are in the same order as the results appear in the introduction.

\section{Proof of Proposition~\ref{product-forest}}

In this section, we give two proofs of  Proposition~\ref{product-forest}.

In the first proof we will use the recursion
$$F(G)=F(G-e)+F(G/e),$$
where $G-e$ is the graph obtained from $G$ by deleting the edge $e$, and $G/e$ is the graph obtained from $G$ by contracting the edge $e$. This latter operation means that we replace the end vertices $u,v$ of $e$ by a new vertex $w$, and for a vertex $s\neq u,v$ we add as many edges between $s$ and $w$ as it goes between $s$ and the set $\{u,v\}$ in the graph $G$, and if there were $k$ edges going between $u$ and $v$ in $G$, then we add $k-1$ loops to the vertex $w$ in $G/e$. The above recursion simply counts the number of forests based on the property that a forest contains the edge $e$ or not. Note that the contraction may produce multiple edges so we necessarily  work in the class of graphs with multiple edges. A forest cannot contain a loop so we can even delete them from the contraction.

\begin{proof}[Proof of Proposition~\ref{product-forest}]
This can easily be proved by induction using the identity \\ $F(G)=F(G-e)+F(G/e)$. If $e=(u,v)$, then

\begin{align*}
F(G)=F(G-e)+F(G/e)&\leq (d_u-1+1)(d_v-1+1)\prod_{w\neq u,v}(d_w+1)\\
&\ +(d_u+d_v-2+1)\prod_{w\neq u,v}(d_w+1)\\
&\leq (d_u+1)(d_v+1)\prod_{w\neq u,v}(d_w+1).
\end{align*}
\end{proof}

\begin{proof}[Second proof]
For a graph $G$ and an orientation $\mathcal{O}$ of the edges, the score vector of $\mathcal{O}$ is simply the out-degree sequence of this orientation. It is known that the number of different score vectors is exactly the number of forests of $G$. This is an unpublished result of R. Stanley, a bijective proof can be found in \cite{KlWi}. Since the out-degree of a vertex $u$ is between $0$ and $d_u$, the number of different score vectors is at most $\prod_{u\in V(G)}(d_u+1)$. 
\end{proof}

\section{Proof of Theorem~\ref{3-regular-forest}}

In this section, we prove Theorem~\ref{3-regular-forest}. Since $a(G)\leq F(G)\leq 2^{e(G)}$ it is enough to prove that $\lim_{n\to \infty}a(G_n)^{1/v(G_n)}=2\sqrt{2}$. In fact, we will prove a slightly stronger theorem. For this, we need the concepts of weakly induced forest and  broken cycle. Figure 2 may help to understand these concepts.

\begin{Def}
Let us label the edges of the graph $G$ with numbers from $1$ to $|E(G)|$. A broken cycle is an edge set that we obtain from a cycle by deleting the edge with the largest label. Let $c_k(G)$ be the number of edge sets with exactly $k$ edges that do not contain any broken cycle. (Note that these edge sets must be forests, since they cannot contain cycles.)

\end{Def}

\begin{Def}
A set $S\subseteq E(G)$ is called a weakly induced forest if it contains no cycle, and the connected components determined by $S$ induces exactly the edges of $S$, all other edges are going between the connected components. Note that the vertex set of a weakly induced forest is the vertex set of the original graph $G$, that is, $V(G)$. Let $F_{wi}(G)$ be the number of weakly induced forests.
\end{Def}

\begin{figure}[h!] \label{broken cycles picture}
\begin{center}
\scalebox{.65}{\includegraphics{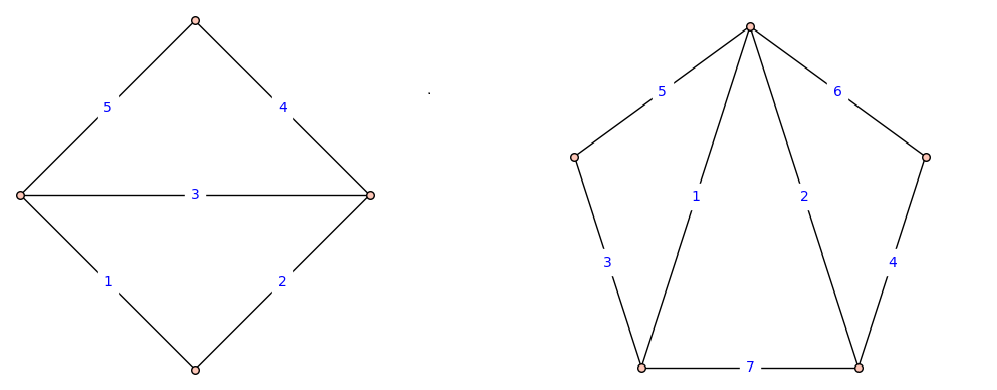}}
\end{center}
\caption{In the left graph, the broken cycles are edge sets $\{1,2\}$, $\{3,4\}$, $\{1,2,4\}$; the weakly induced forests are $\varnothing$, $\{1\}$, $\{2\}$, $\{3\}$, $\{4\}$, $\{5\}$, $\{1,4\}$, $\{1,5\}$, $\{2,4\}$, $\{2,5\}$. In the right graph, the broken cycles are $\{1,3\}$, $\{1,2\}$, $\{2,4\}$, $\{2,3,5\}$, $\{1,4,6\}$, $\{3,4,5,6\}$; some of the induced forests are $\{1,4\}$, $\{3,7\}$, $\{4,5,7\}$, $\{5,6,7\}$.}
\end{figure}

\begin{Lemma} \label{two inequalities} For any graph $G$, we have
$$F_{wi}(G)\leq a(G)\leq F(G).$$
\end{Lemma}

\begin{proof}
The proof is based on the fact that $a(G)$ is the number of edge subsets of $E(G)$ without a broken cycle. This follows from the following well-known facts. Let $\ch(G,q)$ be the chromatic polynomial of the graph $G$, this polynomial counts the number of proper colorings of the graph $G$ when we color the vertices of the graph with $q$ colors (it is allowed that some colors are not used), see for instance \cite{Read}. It is also known \cite{Stan} that $|\ch(G,-1)|=a(G)$. Furthermore, $\ch(G,q)=\sum_{k=0}^{n-1}(-1)^kc_k(G)q^{n-k}$, 
see \cite{Whit}. So $a(G)=\sum_{k=0}^{n-1}c_k(G)$ is the number of edge subsets of $E(G)$ without a broken cycle. From this it is immediately clear that $a(G)\leq F(G)$. It is also clear that a weakly induced forest does not contain any broken cycle no matter what the labeling is since it does not contain a path that can be obtained by deleting an edge from a cycle. Hence $F_{wi}(G)\leq a(G)\leq F(G)$.
\end{proof}

\begin{Rem} This remark outlines another proof of Lemma~\ref{two inequalities} via shattering sets. Its sole purpose is to share this unusual proof, the reader should feel free to skip this remark.

The following proof of Lemma~\ref{two inequalities} is based on an observation of Kozma and Moran \cite{KoMo}. For a set $X$ and a set system $\mathcal{S}\subseteq 2^{X}$ let
$$\mathrm{str}(\mathcal{S})=\{ Y\subseteq X\ |\ \forall A\subseteq Y\  \ \exists S\in \mathcal{S}\ \ A=Y\cap S\}$$
and
$$\mathrm{sstr}(\mathcal{S})=\{ Z\subseteq X\ |\ \exists B\subseteq X\setminus Z\ \  \forall A\subseteq Z\ \ \ A\cup B\in \mathcal{S}\}.$$
The elements of the set system $\mathrm{str}(\mathcal{S})$ are the shattered sets of $\mathcal{S}$, and elements of the set system $\mathrm{sstr}(\mathcal{S})$ are the strongly shattered sets of $\mathcal{S}$. It is known that
$$|\mathrm{sstr}(\mathcal{S})|\leq |\mathcal{S}|\leq |\mathrm{str}(\mathcal{S})|.$$
Now let $X=E(G)$ and let us fix an orientation of the edges. Then every orientation corresponds to a subset of $E(G)$, namely to the edge set where the orientation differs from the fixed orientation. Now following Kozma and Moran let $\mathcal{S}$ be the family of acyclic orientations. Then $Y\subset E(G)$ is shattered if no matter how we orient the edges of $Y$ we can extend it to an acyclic orientation of $G$. It is easy to see that these are exactly the forests of $G$: first of all, it cannot contain a cycle, because then by orienting the cycle we cannot extend it to an acyclic orientation. Secondly, if we orient a forest somehow, then we can orient the rest of the edges according to some topological order that is compatible with the orientation of the forest. A set $Z\subset E(G)$ is strongly shattered if we can orient the rest of the edges in a way that no matter how the edges of $Z$ are oriented it will be an acyclic orientation. Again it is easy to see that these edge sets are exactly the weakly induced forests of $G$: such an edge set cannot contain a cycle or a cycle minus an edge, because otherwise no matter how we orient the rest of the edges we would be able to achieve a cycle by orienting the elements of $Z$. On the other hand, if $Z$ determines a weakly induced forest, then by numbering the connected components of $Z$, and orienting the rest of the edges towards the largest numbers we get an orientation of $E(G)\setminus Z$ that satisfies that no matter how we orient the edges of $Z$ it will yield an acyclic orientation.
\end{Rem}

As a preparation for the proof of Theorem~\ref{3-regular-forest} we add some remarks. The proof uses probabilistic ideas, in particular, the so-called FKG-inequality \cite{FKG,AlSp}. For each subset $S\subseteq E(G)$ we can associate the indicator vector $\omega_S\in \{0,1\}^{E(G)}$:
$$\omega_S(e)=\left\{ \begin{array}{cl} 1 & \mbox{if}\ e\in S, \\ 0 & \mbox{if}\ e\notin S. \end{array} \right.$$
There is a natural partial ordering on the vectors of $\{0,1\}^{E(G)}$: $\omega\leq \omega'$ if for all $e$ we have $\omega(e)\leq \omega'(e)$. On the level of sets this simply means that $\omega_S\leq \omega_{S'}$ if and only if $S\subseteq S'$. A function $f:\{0,1\}^{E(G)}\to \mathbb{R}$ is monotone increasing if $f(\omega)\leq f(\omega')$ whenever $\omega\leq \omega'$. A function $g:\{0,1\}^{E(G)}\to \mathbb{R}$ is monotone decreasing if $g(\omega)\geq g(\omega')$ whenever $\omega\leq \omega'$. 

Next let us consider the uniform measure $\mu$ on $\{0,1\}^{E(G)}$, that is, $\mu(\omega_S)=\frac{1}{2^{e(G)}}$ for all $S\subseteq E(G)$. Then $\mu$ trivially satisfies (with equality) the so-called log-supermodularity inequality:
$$\mu(\omega)\mu(\omega')\leq \mu(\omega \vee \omega')\mu(\omega \wedge \omega'),$$
where $\omega \vee \omega'$ is the vector such that $(\omega \vee \omega')(e)=\max(\omega(e),\omega'(e))$, and 
$\omega \wedge \omega'$ is the vector such that $(\omega \wedge \omega')(e)=\min(\omega(e),\omega'(e))$. The so-called FKG-inequality \cite{FKG} asserts that if the random variables $X,Y:\{0,1\}^{E(G)}\to \mathbb{R}_{\geq 0}$  are both monotone increasing or both monotone decreasing and $\mu$ is log-supermodular, then
$$\mathbb{E}(XY)\geq \mathbb{E}(X)\mathbb{E}(Y).$$
Since the product of non-negative monotone decreasing random variables is also monotone decreasing, we get that
$$\mathbb{E}\left[\prod_{i=1}^kX_i\right]\geq \prod_{i=1}^k\mathbb{E}[X_i]$$
if every $X_i$ is non-negative monotone decreasing random variable. With this preparation, we can prove Theorem~\ref{3-regular-forest}.

\begin{proof}[Proof of Theorem~\ref{3-regular-forest}]
We will actually prove that $\lim_{n\to \infty}F_{wi}(G_n)^{1/v(G_n)}=2\sqrt{2}$. By the above lemma it implies that $\lim_{n\to \infty}a(G_n)^{1/v(G_n)}=2\sqrt{2},$
and \\
$\lim_{n\to \infty}F(G_n)^{1/v(G_n)}=2\sqrt{2}$.

Let us consider a subset $A\subseteq E$ chosen uniformly at random from all possible $2^{e(G)}$ subsets. Then
$$\mathbb{P}(A\ \mbox{is a weakly induced forest})=2^{-3n/2}F_{wi}(G).$$
Let $C_1,C_2,\dots ,C_k$ be the connected components of $A$. Note that $C_j$ might be a single vertex, or it might contain a cycle.
For a fixed vertex $v$ let $X_v:\{0,1\}^{E(G)}\to \mathbb{R}$ be the indicator variable that the vertex $v$ is in a weakly induced tree, that is, it is a tree whose connected component contains only the edges of the tree. 
In other words, if $v\in C_j$ for some $j$, then $X_v(\omega_A)=1$ if $C_j$ is a tree that only contains the edges of $G[V(C_j)]$, otherwise it is $0$. In particular, it is $0$ if it contains a cycle, or $G[V(C_j)]$ contains an edge that is not in $A$. 

The set $A$ is weakly induced forest if and only if $X_v(\omega_A)=1$ for all $v\in V(G)$. Hence
$$\mathbb{P}(A\ \mbox{is a weakly induced forest})=\mathbb{E}\left[\prod_{v\in V}X_v\right].$$
Observe that $X_v$ are all monotone decreasing functions (a subset of a weakly induced tree is also weakly induced), and so by the FKG-inequality we get that
$$\mathbb{E}\left[\prod_{v\in V}X_v\right]\geq \prod_{v\in V}\mathbb{E}[X_v].$$
Here $\mathbb{E}[X_v]$ is the probability that $v$ is in a weakly induced tree. Suppose that $g(G)\geq 2k+1$, and let us choose $R=k-1$, then the $R$-neighborhood of any vertex in $G$ is an induced  tree. The probability that $v$ is in a weakly induced tree is clearly bigger than the probability that in $A$ there is no path between $v$ and a vertex at distance $R$. Next we examine this probability. 

Before estimating the above probability, let us consider the case when we have a  $2$-ary tree of depth $t$, that is the root vertex has degree $2$ and all other non-leaf vertices have degree $3$, and all leaves are of distance $t$ from the root vertex.

For a $2$-ary tree of depth $t$ let us consider a random subset of edges. The probability $p_t$ that the root vertex is connected to some leaf vertex in this random subset satisfies the recursion $p_t=p_{t-1}-\frac{1}{4}p_{t-1}^2$. Clearly, $p_t$ is a monotone decreasing sequence and the limit $q$ must satisfy $q=q-\frac{1}{4}q^2$, so $q=0$. 

The $R$-neighborhood of a vertex  $v$ in the graph $G$ is not exactly a $2$-ary tree, because $v$ has degree $3$ unlike the root vertex of a $2$-ary tree which has degree $2$. Still, we can upper bound the probability that in $A$ there is a path between $v$ and a vertex at distance $R$
by $3p_{R-1}$. Hence $\mathbb{E}[X_v]\geq 1-3p_{R-1}$. So we have
$$\mathbb{P}(A\ \mbox{is a weakly induced forest})\geq (1-3p_{R-1})^{v(G)}.$$
Then using the trivial upper bound and the lower bound obtained now we get that
$$2\sqrt{2}\geq F_{wi}(G)^{1/v(G)}\geq 2\sqrt{2}(1-3p_{R-1}).$$
Since $p_{R}\to 0$ as $n\to\infty$ we get that
$$\lim_{n\to \infty}F_{wi}(G_n)^{1/v(G_n)}=2\sqrt{2}.$$
\end{proof}

\begin{Rem} Since subgraphs free of broken cycles are already decreasing it would have been enough to consider $a(G)$, but from the point of view of the theorem, it was a bit more convenient and natural to use weakly induced forests.

\end{Rem}

\section{Proof of Proposition~\ref{average-forest}}

In this section we prove Proposition~\ref{average-forest}.

\begin{proof}[Proof of Proposition~\ref{average-forest}]
Let $n$ be the number of vertices and let $m=\overline{d}n/2$ be the number of edges. Since a forest has at most $n-1$ edges we have
$$F(G)\leq \sum_{r=0}^{n-1}\binom{m}{r}\leq \exp\left(m H\left(\frac{n-1}{m}\right)\right)\leq \exp\left(m H\left(\frac{n}{m}\right)\right),$$
where we used the fact that $H(x)$ is monotone increasing for $0<x<1/2$.
Hence
$$\frac{1}{v(G)}\ln F(G)\leq \frac{\overline{d}}{2}H\left(\frac{2}{\overline{d}}\right).$$

\end{proof}

\section{Proof of Theorem~\ref{4-regular-forest}}

In this section we prove Theorem~\ref{4-regular-forest}. The following lemma is not crucial, but will turn out to be useful at some point to avoid certain technical difficulties.

\begin{Lemma} \label{special cover} Let $G$ be a graph and $e=(u,v)\in E(G)$. Let us consider the graph $H$ obtained from $G\cup G$ in such a way that we delete the two copies $(u_1,v_1)$ and $(u_2,v_2)$ of the edge $e$ and add the edges $(u_1,v_2)$ and $(u_2,v_1)$. Then
$$F(G)^2\leq F(H).$$
\end{Lemma}

\begin{Rem} If $G$ is connected and $e=(u,v)\in E(G)$ is not a cut edge, then $H$ is connected too. Note that $H$ is a very special $2$-cover of $G$.
\end{Rem}

\begin{proof} Let
$$F_1=\left\{S\subseteq E(G)\setminus \{e\}\ |\ S\mbox{\ is a forest}\right\}\ \ \mbox{and}\ \ F_2=\left\{S\subseteq E(G)\setminus \{e\}\ |\ S\cup \{e\}\mbox{\ is a forest}\right\}.$$
Then $|F_1|\geq |F_2|$ and $F(G)=|F_1|+|F_2|$. Set $|F_1|=f_1$ and $|F_2|=f_2$. Note that $F(H)=3f_1^2+(f_1^2-(f_1-f_2)^2)=3f_1^2+2f_1f_2-f_2^2$ since if we add at most one of the edges of $(u_1,v_2)$ and $(u_2,v_1)$, then there are $f_1^2$ ways to add a subset of the edges such that it will be a forest, and if we add both edges, then the only bad case that we add sets $S_1,S_2\in F_1\setminus F_2$ in the two copies of $G$. Since $3f_1^2+2f_1f_2-f_2^2\geq (f_1+f_2)^2$ we are done.

\end{proof}

Now we are ready to prove Theorem~\ref{4-regular-forest}.

\begin{proof}[Proof of Theorem~\ref{4-regular-forest}]
First, let us assume that  $G$ is connected.
The idea is to bound the number of forests according to the number of edges. If the number of edges of the forest is at most $(1-\varepsilon)n$, then the number of forests is at most
$$\sum_{k=0}^{(1-\varepsilon)n}\binom{2n}{k}\leq \exp\left(2nH\left(\frac{1-\varepsilon}{2}\right)\right).$$
If the number of edges is at least $(1-\varepsilon)n$, then we can get it from a spanning tree by deleting at most $\varepsilon n$ edges. Hence the number of such forests is at most
$$\tau(G)\sum_{k=0}^{\varepsilon n}\binom{n-1}{k}\leq \tau(G)\exp(nH(\varepsilon)),$$
where $\tau(G)$ is the number of spanning trees.

We will use the theorem of McKay \cite{McKay1} claiming that the number of spanning trees of a $d$-regular graph is at most
$$\tau(G)\leq \frac{c_d\ln n}{n}\left(\frac{(d-1)^{d-1}}{(d^2-2d)^{d/2-1}}\right)^n.$$
For us $d=4$, so
$$\tau(G)\leq \frac{c_4\ln n}{n}\left(\frac{27}{8}\right)^n.$$
Hence
$$F(G)\leq \exp\left(2nH\left(\frac{1-\varepsilon}{2}\right)\right)+\frac{c_4\ln n}{n}\left(\frac{27}{8}\right)^n\exp(nH(\varepsilon)).$$
By choosing $\varepsilon=0.04$ we get that $F(G)\leq C\cdot 3.994^n$, where $C$ is some absolute constant.
Next, we show that this statement is true for all $4$-regular connected graphs without $C$, that is, $F(G)\leq 3.994^n$. Let
$$M=\sup_{G\in \mathcal{G}^c_4} \frac{F(G)}{3.994^{v(G)}},$$
where the supremum is taken over all $4$-regular connected  graphs. We know that $M\leq C$. Let $G$ be an arbitrary $4$-regular connected graph. Now let $H$ be the special $2$-cover described in Lemma~\ref{special cover} such that $H$ is connected too. Then
$$F(G)^2\leq F(H)\leq M\cdot 3.994^{v(H)}=M\cdot 3.994^{2v(G)},$$
whence $F(G)\leq \sqrt{M}\cdot 3.994^{v(G)}$. Since $G$ was arbitrary we get that $M\leq \sqrt{M}$, hence $M\leq 1$.
Hence $F(G)^{1/v(G)}\leq 3.994$ for all $4$-regular connected graphs. Since $F(\bigcup G_i)=\prod F(G_i)$, the same inequality is true for disconnected graphs.

\end{proof}

\section{Proof of Theorem~\ref{forest-cover}}

In this section we prove Theorem~\ref{forest-cover}. We first need a lemma.

\begin{Lemma} \label{equivalent inequalities}
Let $S$ be a finite set, and let $\mu$ be a probability distribution on the subsets of $S$. Let $x,y$ be fixed elements of $S$. Let $\textbf{S}$ be a random subset of $S$ according to the distribution $\mu$. Then the following inequalities are equivalent
$$\P_{\mu}(x,y\in \textbf{S})\leq \P_{\mu}(x\in \textbf{S})\P_{\mu}(y\in \textbf{S}),$$
$$\P_{\mu}(x,y\notin \textbf{S})\leq \P_{\mu}(x\notin \textbf{S})\P_{\mu}(y\notin \textbf{S}),$$
$$\P_{\mu}(x,y\in \textbf{S})\P_{\mu}(x,y\notin \textbf{S})\leq \P_{\mu}(x\in \textbf{S},y\notin \textbf{S})\P_{\mu}(x\notin \textbf{S},y\in \textbf{S}).$$
\end{Lemma}

\begin{proof}
We prove the equivalence of the first and third inequalities, the rest is similar. 
$$\P_{\mu}(x\in \textbf{S})=\P_{\mu}(x\in \textbf{S},y\in \textbf{S})+\P_{\mu}(x\in \textbf{S},y\notin \textbf{S})$$
and
$$\P_{\mu}(y\in \textbf{S})=\P_{\mu}(x\in \textbf{S},y\in \textbf{S})+\P_{\mu}(x\notin \textbf{S},y\in \textbf{S}).$$
Furthermore,
$$\P_{\mu}(x\in \textbf{S},y\in \textbf{S})=\P_{\mu}(x\in \textbf{S},y\in \textbf{S})\cdot 1=$$
$$\P_{\mu}(x\in \textbf{S},y\in \textbf{S})(\P_{\mu}(x\in \textbf{S},y\in \textbf{S})+\P_{\mu}(x\in \textbf{S},y\notin \textbf{S})+\P_{\mu}(x\notin \textbf{S},y\in \textbf{S})+\P_{\mu}(x\notin \textbf{S},y\notin \textbf{S})).$$
Now writing the above identities into $\P_{\mu}(x,y\in \textbf{S})\leq \P_{\mu}(x\in \textbf{S})\P_{\mu}(y\in \textbf{S}),$ and subtracting the identical terms we get the third inequality.
\end{proof}

In the forthcoming proof, we will apply Lemma~\ref{equivalent inequalities} and Conjecture~\ref{correlation-conjecture} to the set $S=E(H)$ and probability distribution $\mu$ which takes value $0$ on the non-forest subsets, and uniform on the forests.

The other tool that we need is the recursion
$$F(G)=F(G-e)+F(G/e)$$
that we already used in the proof of Proposition~\ref{product-forest}.
As we already noted the contraction may produce multiple edges so we necessarily  work in the class of graphs with multiple edges. Conjecture~\ref{correlation-conjecture} is expected to be true for graphs with multiple edges, in fact, it is expected to be true even for weighted graphs. 

Now we are ready to prove Theorem~\ref{forest-cover}.

\begin{proof}[Proof of Theorem~\ref{forest-cover}]
We prove the statement by induction on the number of edges. When $G$ is the empty graph on $n$ vertices, the claim is trivial. Let $e=(u,v)\in E(G)$, and let $e_1$ and $e_2$ be the $2$-lifts of $e$ in a $2$-cover $H$. For the sake of simplicity we also denote by $e_1$ and $e_2$ the $2$-lifts of $e$ in $G\cup G$. We decompose $F(H)$ according to the cases whether a forest contains $e_1$ and/or $e_2$:
$$F(H)=F_{e_1,e_2}(H)+F_{\overline{e}_1,e_2}(H)+F_{e_1,\overline{e}_2}(H)+F_{\overline{e}_1,\overline{e}_2}(H),$$
where the first term means that we count the number of forests containing both $e_1,e_2$, the second term counts the number of forests containing $e_2$, but not $e_1$, etc.
Similarly, we have
$$F(G\cup G)=F_{e_1,e_2}(G\cup G)+F_{\overline{e}_1,e_2}(G\cup G)+F_{e_1,\overline{e}_2}(G\cup G)+F_{\overline{e}_1,\overline{e}_2}(G\cup G).$$
Now observe that the terms $F_{\overline{e}_1,\overline{e}_2}(H)$ and $F_{\overline{e}_1,\overline{e}_2}(G\cup G)$ count the number of forests in $2$-covers of $G-e$, and by induction
$$F_{\overline{e}_1,\overline{e}_2}(H)\geq F_{\overline{e}_1,\overline{e}_2}(G\cup G).$$
Similarly, $F_{e_1,e_2}(H)=F(H/\{e_1,e_2\})$ and $H/\{e_1,e_2\}$ is isomorphic to a $2$-cover of $G/e$. Hence
$$F_{e_1,e_2}(H)\geq F_{e_1,e_2}(G\cup G).$$
Observe that by symmetry we have $F_{\overline{e}_1,e_2}(H)=F_{e_1,\overline{e}_2}(H)$. Let $\mathbf{F}$ be a random forest of $H$ chosen uniformly, and $\mathbb{P}_H$ be the  corresponding probability distribution. 
Note that with our previous notation we have 
$$\P_H(e_1\in \textbf{F},e_2\notin \textbf{F})=\frac{F_{e_1,\overline{e}_2}(H)}{F(H)}\ \ \ \mbox{and}\ \ \ \P_H(e_1\notin \textbf{F},e_2\in \textbf{F})=\frac{F_{\overline{e_1},e_2}(H)}{F(H)}.$$
Lemma~\ref{equivalent inequalities} shows that the negative correlation inequality of Conjecture~\ref{correlation-conjecture}, namely,
$$\P_H(e,f\in \textbf{F})\leq \P_H(e\in \textbf{F})\P_H(f\in \textbf{F})$$
is equivalent to
$$\P_H(e\in \textbf{F},f\notin \textbf{F})\P_H(e\notin \textbf{F},f\in \textbf{F})\geq \P_H(e\in \textbf{F},f\in \textbf{F})\P_H(e\notin \textbf{F},f\notin \textbf{F}).$$
In the following computation, we will apply this inequality to $e=e_1$ and $f=e_2$.
Then
\begin{align*}
F_{\overline{e_1},e_2}(H)^2&=F_{\overline{e}_1,e_2}(H)\cdot F_{e_1,\overline{e}_2}(H)\\
                           &=F(H)^2\P_H(e_1\in \textbf{F},e_2\notin \textbf{F})\P_H(e_1\notin \textbf{F},e_2\in \textbf{F})\\
													&\geq F(H)^2\P_H(e_1\in \textbf{F},e_2\in \textbf{F})\P_H(e_1\notin \textbf{F},e_2\notin \textbf{F})\\
                          &=F_{e_1,e_2}(H)F_{\overline{e}_1,\overline{e}_2}(H)\\
													&\geq F_{e_1,e_2}(G\cup G)F_{\overline{e}_1,\overline{e}_2}(G\cup G)\\
													&=F_{e_1}(G)F_{e_2}(G)F_{\overline{e}_1}(G)F_{\overline{e}_2}(G)\\
													&=F_{\overline{e}_1,e_2}(G\cup G)\cdot F_{e_1,\overline{e}_2}(G\cup G)\\
													&=F_{\overline{e}_1,e_2}(G\cup G)^2.
\end{align*}
Hence $F_{\overline{e}_1,e_2}(H)\geq F_{\overline{e}_1,e_2}(G\cup G)$. Putting together the $4$ inequalities we get that $F(H)\geq F(G\cup G)$.

\end{proof}

\section{Proof of Theorem~\ref{regular upper bounds}}

In this section, we give a new upper bound on the number of forests in regular graphs. We use some basic results from spectral graph theory such as the matrix-tree theorem and the expression of the number of closed walks as a power sum of the eigenvalues of the adjacency matrix. All these results can be found in the books of Brouwer and Haemers \cite{BrHa} and Godsil and Royle \cite{GoRo}.

\begin{Def}
Let $G$ be a graph with edge weights $w:E(G)\to \mathbb{R}$. Let $L(G,w)$ be the $|V|\times |V|$ matrix defined as follows:
$$L(G,w)_{i,j}=\left\{\begin{array} {cl}
\sum_{e: i\in e}w_e & \mbox{if}\ i=j, \\
-w_e &\mbox{if}\ e=(i,j)\in E(G),\\
0 &\mbox{if}\ (i,j)\notin E(G).
\end{array}
\right.$$
The matrix $L(G,w)$ is the weighted Laplacian matrix of the graph $G$.
\end{Def}

\begin{Lemma}[Kirchhoff's matrix-tree theorem \cite{Kirc,BrHa,GoRo}] \label{Kirchhoff}
Let $G$ be a graph with edge weights $w:E(G)\to \mathbb{R}$. Let $L(G,w)$ be its weighted  Laplacian matrix. Let $L_0(G,w)$ be the matrix obtained from $L(G,w)$ by deleting the first row and column. Then
$$\det L_0(G,w)=\sum_{T\in \mathcal{T}(G)}\prod_{e\in E(T)}w_e,$$
where $\mathcal{T}(G)$ is the set of spanning trees of $G$.

\end{Lemma}

Let $G$ be a $d$-regular graph with Laplacian matrix $L(G,\underline{1})$. Let us add one more vertex to $G$, and connect it to all vertices with an edge of weight $\alpha$. Let $G_{\alpha}$ be the new graph. The original edges have weight $1$, and a weight of a spanning tree of the graph $G_{\alpha}$ is the product of the weights of the edges in the spanning tree.
Then the weighted sum of the spanning trees of $G_{\alpha}$ can be computed as a weighted sum of the forests of the original graph $G$ as follows. The total weight of spanning trees in $G_{\alpha}$, which correspond to a forest $F$ in $G$ with connected components $F_1,F_2,\dots ,F_k$, is
$$w_{\alpha}(F)=\alpha^k\prod_{i=1}^k|F_i|.$$
Indeed, once we have a forest of $G$ we can create $\prod_{i=1}^k|F_i|$ spanning trees of $G_{\alpha}$ by connecting one of the vertices of each component to the new vertex. Each such spanning tree has a weight $\alpha^k$.

Let $k(F)$ denote the number of connected components of a forest $F$. By the above discussion we get that for each $\alpha>0$ we have $w_{\alpha}(F)(1/\alpha)^{k(F)}\geq 1$.
Let us use the matrix-tree theorem (Lemma~\ref{Kirchhoff}):
we have $L_0(G_{\alpha},w)=L(G,\underline{1})+\alpha I$, where $I$ is the identity matrix. Then
$$S:=\sum_Fw_{\alpha}(F)=\det(L(G,\underline{1})+\alpha I)=\prod_{i=1}^n(\lambda_i+\alpha),$$
where $\lambda_i$ are the  eigenvalues of the matrix $L(G,\underline{1})$. If $d=\mu_1\geq \mu_2\geq \dots \geq \mu_n$ are the eigenvalues of the adjacency matrix of the graph $G$, then $\lambda_i=d-\mu_i$. Hence
$$\frac{1}{n}\ln S=\frac{1}{n}\sum_{i=1}^n\ln(\lambda_i+\alpha)=\frac{1}{n}\sum_{i=1}^n\ln(d-\mu_i+\alpha).$$
Now we can estimate this sum as follows. In the following computation $\mu_{KM}$ is the Kesten-McKay measure \cite{McKay1,McKay2}. Its explicit form is given by the density function
$$\frac{d\sqrt{4(d-1)-x^2}}{2\pi(d^2-x^2)}\cdot 1_{(-2\sqrt{d-1},2\sqrt{d-1})}.$$
Its speciality is that
$$W_k(\mathbb{T}_d,o):=\int x^k\, d\mu_{KM}(x)$$
is equal to the number of closed walks of length $k$ in the infinite $d$-regular tree from a fixed root vertex $o$.
The quantity $W_k(G)=\sum_{i=1}^n\mu_i^k$ counts the number of closed walks of length $k$ in the graph $G$. The following standard argument shows that $W_k(G)\geq nW_k(\mathbb{T}_d,o)$ since $\mathbb{T}_d$ is the universal cover of any $d$-regular graph $G$. We show that for any vertex $v$, the number of closed walks $W_{k}(G,v)$ of length $k$ starting and ending at vertex $v$ is at least as large as the number of closed walks starting and ending at some root vertex of the infinite $d$-regular tree $\mathbb{T}_d$. Let us consider the following infinite $d$-regular tree, its vertices are labeled by the walks starting at the vertex $v$ which never steps immediately back to a vertex, where it came from. Such walks are called non-backtracking walks. For instance, in the depicted graph below $149831$ is such a walk, but $1494$ is not a backtracking walk since after $9$ we immediately stepped back to $4$. We connect two non-backtracking walks in the tree if one of them is a one-step extension of the other.

\begin{figure}[h!] \label{universal cover picture}
\begin{center}
\scalebox{.16}{\includegraphics{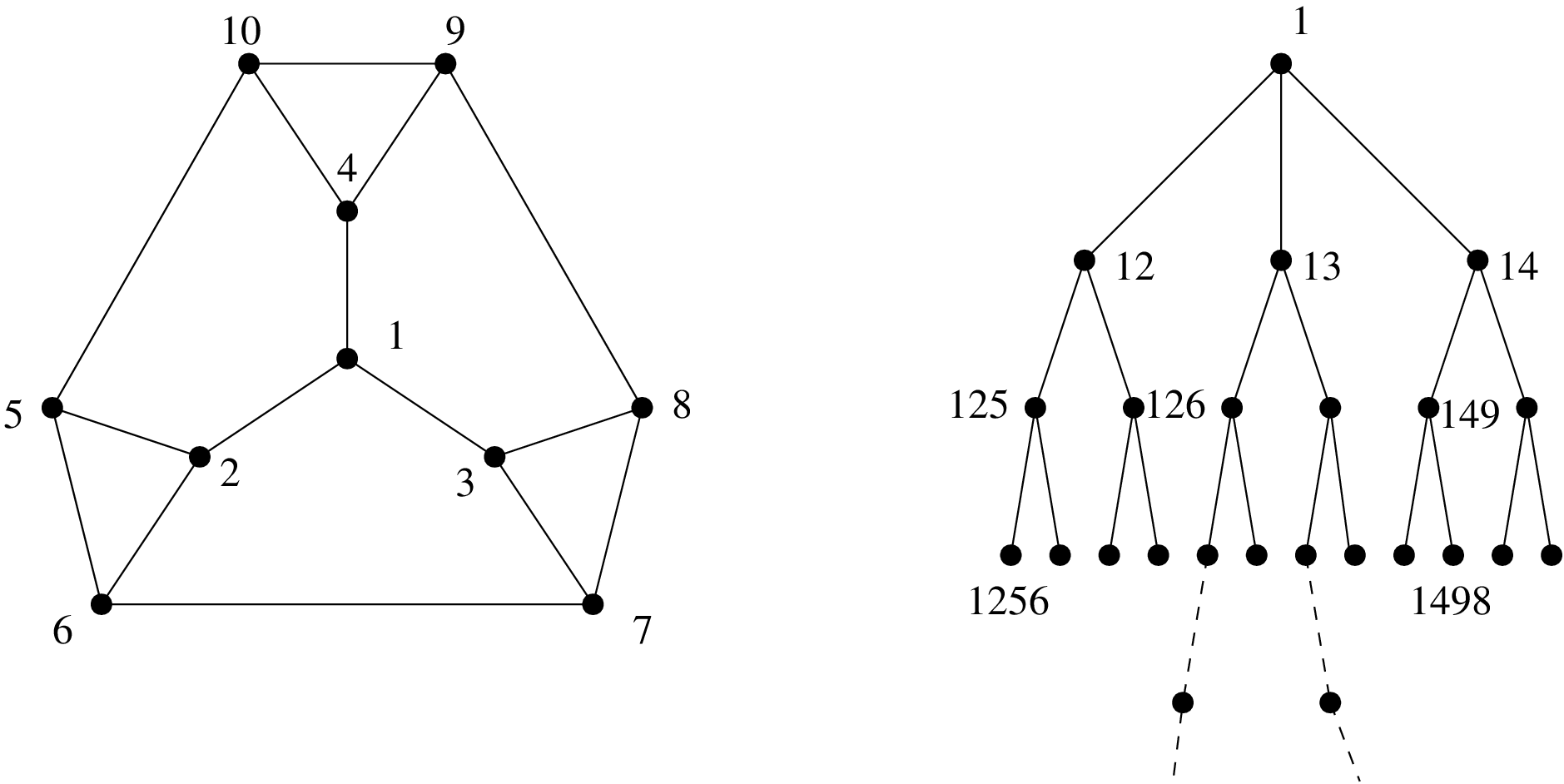}}
\end{center}
\end{figure}

Note that every closed walk in the tree corresponds to a closed walk in the graph $G$: for instance, for the depicted graph the walk $1,14,149,14,1$ in the tree corresponds to the walk $1,4,9,4,1$ in the graph.  On the other hand, there are closed walks in the graph $G$, like $149831$, which are not closed anymore in the tree. This argument shows that $W_k(G,v)\geq W_k(\mathbb{T}_d,o)$ for all $v\in V(G)$. Consequently, $W_k(G)\geq nW_k(\mathbb{T}_d,o)$.

Then
\begin{align*}
\frac{1}{n}\sum_{i=1}^n\ln(d-\mu_i+\alpha)&=\ln(d+\alpha)+\frac{1}{n}\sum_{i=1}^n\ln\left(1-\frac{\mu_i}{d+\alpha}\right)\\
&=\ln(d+\alpha)+\frac{1}{n}\sum_{i=1}^n\sum_{k=1}^{\infty}-\frac{1}{k}\left(\frac{\mu_i}{d+\alpha}\right)^k\\
&=\ln(d+\alpha)-\sum_{k=1}^{\infty}\frac{1}{k(d+\alpha)^k}\frac{1}{n}\sum_{i=1}^n\mu_i^k\\
&=\ln(d+\alpha)-\sum_{k=1}^{\infty}\frac{1}{k(d+\alpha)^k}\frac{W_k(G)}{n}\\
&\leq \ln(d+\alpha)-\sum_{k=1}^{\infty}\frac{1}{k(d+\alpha)^k}W_k(\mathbb{T}_d,o)\\
&=\ln(d+\alpha)-\sum_{k=1}^{\infty}\frac{1}{k(d+\alpha)^k}\int x^k\, d\mu_{KM}(x)\\
&=\ln(d+\alpha)+\int \ln\left(1-\frac{x}{d+\alpha}\right)\, d\mu_{KM}(x)\\
&=\int \ln(d+\alpha-x)\, d\mu_{KM}(x).
\end{align*}
Hence
$$S\leq \exp\left(n \int \ln(d+\alpha-x)\, d\mu_{KM}(x)\right).$$
When $\alpha=1$, then $w_1(F)\geq 1$, and  we get back the result of Kahale and Schulman (actually they used almost the same argument for acyclic orientations instead of forests). When $\alpha =1/2$, then we get that  the number of forests without isolated vertices, denoted by $F_1(G)$, satisfies
$$F_1(G)\leq \exp\left(n \int \ln\left(d+\frac{1}{2}-x\right)\, d\mu_{KM}(x)\right)<d^n,$$
since $w_{1/2}(F)\geq 1$ for forests without isolated vertices.

\begin{Rem} One can explicitly compute the above integrals using a theorem of McKay \cite{McKay1}. For $|\gamma|<\frac{1}{2\sqrt{d-1}}$ let
$$J_d(\gamma)=\int \ln(1-\gamma x) \, d\mu_{KM}(x).$$
Let
$$\eta=\frac{1-(1-4(d-1)\gamma^2)^{1/2}}{2(d-1)\gamma^2}.$$
Then
$$J_d(\gamma)=-\ln \left(\eta \left(\frac{d-\eta}{d-1}\right)^{(d-2)/2}\right).$$
Clearly, one needs to use that
$$\int \ln(d+\alpha-x)\, d\mu_{KM}(x)=\ln(d+\alpha)+J_d\left(\frac{1}{d+\alpha}\right).$$
This is how we got the explicit bound in Theorem~\ref{KS-bound} that does not appear in the original paper of Kahale and Schulman \cite{KaSc}.
\end{Rem}

Now we are ready to improve  the result of Kahale and Schulman \cite{KaSc}.

\begin{proof}
Let $G$ be a $d$-regular graph on $n$ vertices. Then the number of edges is $m=dn/2$. Let $F_1$ be the number of forests of $G$ with at most $cn$ connected components, where $c$ is some constant that we will choose later. Let $F_2$ be the number of forests of $G$ with more than $cn$ connected components. In this latter case the number of edges of the forest is at most $e(F)=n-k(F)\leq (1-c)n$. Then
$$F_1\leq \exp\left(n \int \ln(d+\alpha-x)\, d\mu_{KM}(x)\right) \cdot \left(\frac{1}{\alpha}\right)^{cn}$$
since $F_1\leq \sum_{F}w_{\alpha}(F)\left(\frac{1}{\alpha}\right)^{k(F)}\leq \sum_{F}w_{\alpha}(F)\left(\frac{1}{\alpha}\right)^{cn}$.
For $F_2$ we use the trivial bound based on the fact that such a forest has at most $(1-c)n$ edges.
$$F_2\leq \sum_{k=0}^{(1-c)n}\binom{m}{k}\leq \exp\left(n\frac{d}{2}H\left(\frac{2(1-c)}{d}\right)\right).$$
Hence
$$F(G)\leq \exp\left(n \int \ln(d+\alpha-x)\, d\mu_{KM}(x)\right) \cdot \left(\frac{1}{\alpha}\right)^{cn}+
\exp\left(n\frac{d}{2}H\left(\frac{2(1-c)}{d}\right)\right).$$
Next, we choose $\alpha$ and $c$ to make the two terms (approximately) the same, then we arrive at the bound
$F(G)\leq 2C_d^n$. If we cannot find such $c$, then we can still use $c=1$ to recover the bound of Kahale and Schulman (KS-bound in the table below) that corresponds to the choice $\alpha=1, c=1$. Besides, one can get rid of the constant $2$ by the trick used in the proof of Theorem~\ref{4-regular-forest}. The value of $C_d$ is given in the table below. As we see we get a much better bound for $d=5,6,7,8$, and a slightly better bound for $d=9$.
\bigskip

\begin{center}
\begin{table}[h!] \label{table forest2}
\caption{Bounds on the number of forests for small $d$}

\begin{tabular}{|c|c|c|c|c|}  \hline
$d$ & new bound $C_d$ & KS-bound & $\alpha$ & $c$ \\ \hline
$5$   & $5.1965$           &   $5.5361$  & $0.3084$ & $0.0739$ \\ \hline
$6$   & $6.3367$           &   $6.5286$  & $0.4482$ & $0.0835$ \\ \hline
$7$   & $7.4290$           &   $7.5235$  & $0.5917$ & $0.0903$ \\ \hline
$8$   & $8.4843$           &   $8.5200$  & $0.7374$ & $0.0955$ \\ \hline
$9$   & $9.5116$           &   $9.5173$  & $0.8844$ & $0.0995$ \\ \hline
\end{tabular}

\end{table}
\end{center}

\end{proof}

\section{Proof of Theorem~\ref{regular-connected}}

In this section we prove Theorem~\ref{regular-connected}. Note that for a bipartite graph $G=(A,B,E)$ there is a strikingly simple argument giving
$$C(G)^{1/v(G)}\leq 2^{d/2}\left(1-\frac{1}{2^d}\right)^{1/2}.$$
Indeed, if we consider the probability that a random edge subset spans a connected graph, then this probability is clearly smaller than the probability that on one side of bipartition none of the vertices are isolated. Let $S$ be a random subset of the edge set. If $B_v$ is the bad event that the vertex $v\in A$ is isolated, then
$$\P(S\ \mbox{spans a connected subgraph})\leq \P(\bigcap_{v\in A} \overline{B_v})=\prod_{v\in A}\P(\overline{B_v})=\left(1-\frac{1}{2^d}\right)^{v(G)/2}.$$
We used the fact that the events $B_v$ for $v\in A$ are independent. This gives the above inequality.

In the general case, there is no independence for all vertices. (Though we can take a large independent set of the vertex set instead of the set $A$.) Nevertheless, we can easily overcome it by using one of
Janson's inequalities. This needs some preparation.
\bigskip

\textbf{Setup of Janson's inequalities.} Let $\Omega$ be a fixed set and let $R$ be a random subset of $\Omega$ by choosing $r\in R$ with probability $p_r$ mutually independently of each other. Let $(A_i)_{i\in I}$ be subsets of $\Omega$ for some index set $I$. Let $B_i$ be the event that $A_i\subseteq R$. Let $X_i$ be the indicator random variable for the event $B_i$. Let
$$X:=\sum_{i\in I}X_i.$$
It is, of course, the number of $A_i\subseteq R$. So the events $\bigcap_{i\in I}\overline{B_i}$ and $X=0$  are identical. For $i,j\in I$ we say that $i\sim j$ if $A_i\cap A_j\neq \emptyset$. Note that if $i\not\sim j$ then this is consistent with our previous notation that $B_i$ and $B_j$ are  independent. Let
$$\Delta=\sum_{i\sim j}\P \left(B_i\cap B_j\right),$$
where the sum is over all ordered pairs, so $\Delta/2$ is the same sum for unordered pairs. Set
$$M=\prod_{i\in I}\P\left(\overline{B_i}\right).$$
This would be the probability of $\bigcap_{i\in I}\overline{B_i}$ if the events $B_i$ were independent. Finally, set
$$\mu=\E X=\sum_{i\in I}\P(B_i).$$
Now we are ready to state Janson's inequalities.

\begin{Th}[Janson inequality \cite{JLR},\cite{AlSp}] \label{janson1} Let $(B_i)_{i\in I},M,\Delta,\mu$ be as above, and assume that $\P(B_i)\leq \varepsilon$ for all $i\in I$. Then
$$M\leq \P \left(\bigcap_{i\in I}\overline{B_i}\right)\leq M \exp \left(\frac{1}{1-\varepsilon}\cdot \frac{\Delta}{2}\right).$$
\end{Th}

\begin{proof}[Proof of Theorem~\ref{regular-connected}]
Now let $R\subseteq E(G)$ be a set chosen uniformly at random. For a vertex $v$ let $A_v$ be the set of edges incident to the vertex $v$. Let $B_v$ be the bad event that $A_v\subseteq R$. If one of $B_v$ occurs, then $E(G)\setminus R$ cannot be connected since the vertex $v$ would be an isolated vertex. We have
$$\P(B_v)=\frac{1}{2^d},$$
and
$$\Delta=\sum_{u\sim v}\P \left(B_u\cap B_v\right)=\frac{nd}{2^{2d-1}}.$$
Then using Janson's inequality with $\varepsilon=\frac{1}{2^d}$ we get that
$$C(G)\leq 2^{nd/2}\P \left(\bigcap_{v\in V}\overline{B_v}\right)\leq 2^{nd/2}\left(1-\frac{1}{2^d}\right)^n\exp\left(\frac{nd}{2^{d}(2^d-1)}\right).$$
Hence
$$C(G)^{1/n}\leq 2^{d/2}\left(1-\frac{1}{2^d}\right)\exp\left(\frac{d}{2^{d}(2^d-1)}\right).$$

\end{proof}

\section{Proof of Theorem~\ref{average-connected}}

In this section we prove Theorem~\ref{average-connected}.

\begin{proof}[Proof of Theorem~\ref{average-connected}]
Let us consider
$$Z_{\mathrm{RC}}(G,q,w)=\sum_{A\subseteq E(G)}q^{k(A)}w^{|A|}.$$
We show that if $0<q\leq 1$ and $w\geq 0$, then
$$Z_{\mathrm{RC}}(G,q,w)\leq q^{v(G)}\left(1+\frac{w}{q}\right)^{e(G)}.$$
Indeed,
$$q^{v(G)}\left(1+\frac{w}{q}\right)^{e(G)}=\sum_{A\subseteq E(G)}q^{v(G)}\left(\frac{w}{q}\right)^{|A|}.$$
So it is enough to show that
$$q^{k(A)}w^{|A|}\leq q^{v(G)}\left(\frac{w}{q}\right)^{|A|},$$
equivalently $1\leq q^{v(G)-k(A)-|A|}$ which is indeed true since $|A|+k(A)\geq v(G)$ and $0<q\leq 1$.
Let us apply this inequality to $q=\frac{\overline{d}-2}{2}$ and $w=1$. Note that $0<q\leq 1$ since $2<\overline{d}\leq 4$. Observe also that just by keeping the connected spanning subgraphs $A$ we have
$$qC(G)\leq  Z_{\mathrm{RC}}(G,q,1)\leq q^{v(G)}\left(1+\frac{1}{q}\right)^{e(G)}.$$
By substituting $q=\frac{\overline{d}-2}{2}$ and evaluating the right-hand side we get that
$$C(G)\leq \frac{2}{\overline{d}-2}\exp\left(v(G)\frac{\overline{d}}{2}H\left(\frac{2}{\overline{d}}\right)\right).$$
\end{proof}

\noindent \textbf{Acknowledgment.} The second author thanks Ferenc Bencs for the useful discussions. The authors are very grateful to the referees for their comments and suggestions. 
\bigskip


\end{document}